\documentclass{amsart}
\usepackage{hyperref,enumerate}
\usepackage{graphicx}
\usepackage{color}

\DeclareMathOperator{\Ric}{Ric}
\DeclareMathOperator{\inj}{inj}

\newtheorem{theorem}{Theorem}[section]
\newtheorem{lemma}[theorem]{Lemma}
\newtheorem{claim}[theorem]{Claim}

\theoremstyle{definition}

\newtheorem{remark}[theorem]{Remark}
\numberwithin{equation}{section}


\title[]{Einstein four-manifolds of positive sectional curvature and ADM mass}

\author{Matthew J. Gursky}
\address{Department of Mathematics \\
         University of Notre Dame\\
         Notre Dame, IN 46556}
\email{\href{mgursky@nd.edu}{mgursky@nd.edu}}

\author{Andrea Malchiodi}
\address{Scuola Normale Superiore \\
         Piazza dei Cavalieri 7, 56126 \\
         Pisa, Italy}
\email{\href{andrea.malchiodi@sns.it}{andrea.malchiodi@sns.it}}

\begin{document}

\begin{abstract}    
	We prove that closed Einstein four-manifolds $(M,g)$ with positive sectional curvature and 
	Euler characteristic and signature satisfying $2\chi(M)-3|\tau(M)|\le4$ 
	must be homothetically isometric to the round $S^4$ or to $\mathbb{CP}^2$ with 
	the Fubini-Study metric. The proof relies on an upper bound on the ADM mass  
	of conformal blow-ups of $(M,g)$, combined with a lower one from \cite{GM}. 
	
\end{abstract}

\maketitle

\section{Introduction and Statement of Main results}

The main result of this paper is the following:

\begin{theorem} \label{MainThm}  Suppose $M$ is a smooth, closed, oriented four-manifold, and $g$ is an Einstein metric with positive sectional curvature.
 If 
\begin{equation}\label{top}
2\chi(M)-3|\tau(M)|\le4,
\end{equation}
then $(M, g)$ is homothetically isometric to the round metric on $S^4$ or the
Fubini-Study metric on $\pm \mathbb{CP}^2$.   
\end{theorem}

It is conjectured that this theorem should hold without the topological restriction (\ref{top}).  Under various topological and pinching conditions there are a number of classification results in the literature, 
assuming for example lower bounds on the sectional curvature (\cite{Berger}, \cite{Liu}, \cite{Costa}), upper curvature bounds (\cite{CT22}, \cite{CS}), positivity 
of the curvature operator (\cite{Tachibana}, \cite{CW}, \cite{Wu}, \cite{WuCorr}), non-negative isotropic curvature (\cite{MW}) and pointwise pinching conditions (\cite{CT18}, \cite{BS}).

The first author and LeBrun 
\cite[Theorem A]{GL} proved that a closed oriented Einstein four-manifold
with non-negative sectional curvature and strictly positive 
intersection form is homothetically isometric to $\mathbb{CP}^2$ with the Fubini--Study
metric.  Theorem B of \cite{GL} gives the obstruction
\begin{equation}\label{eq:gl-obstruction}
9\ge\chi(M)>\frac{15}{4}|\tau(M)|
\end{equation}
for an Einstein metric with non-negative sectional curvature that is
neither self-dual nor anti-self-dual. In particular, uniqueness on
$\mathbb{CP}^2$ under nonnegative sectional curvature is already known.
The inequality \eqref{eq:gl-obstruction} is an improvement of a result of Hitchin \cite{Hitchin74}, 
who proved that for Einstein metrics of positive sectional curvature,  
$\chi(M) \geq \left( \frac{3}{2}\right)^{\frac 32} |\tau(M)|$.

The proof of Theorem \ref{MainThm} uses a very different approach from the results cited above.  First, under the assumptions of Theorem \ref{MainThm}, the work of the first author and LeBrun implies that either $g$ is homothetically isometric to the Fubini-Study metric, or $\chi(M) = 2$ and $\tau(M) = 0$ (see Section \ref{sec:rigidity} for details). Given $(M,g)$ satisfying the assumptions of the theorem, we normalize $g$ by assuming 
\begin{align} \label{normal}
\Ric = 3g.
\end{align}
We then fix a point $p \in M$, and consider the Green's function $G > 0$ of the conformal laplacian $L = \Delta_g - 2$, with pole at $p$.  We normalize $G$ so that near $p$, 
\begin{align} \label{Gnorm}
G = r^{-2} + A + O(r),
\end{align} 
where $r$ is the geodesic distance to $p$. The metric $h = G^2 g$ on $M \setminus \{ p \}$ is scalar flat and asymptotically flat of order two, hence its mass $m(h)$ is well defined.   In our previous work (see \cite[Theorem~1.5]{GM}), we showed that under the assumption $\chi(M) = 2, \tau(M) = 0$, either $g$ is round or the mass satisfies $m(h) \geq 22/3$.   Such an estimate was obtained from integral identities and gradient bounds on the Green's function for the conformal Laplacian, 
combined with the results in \cite{Gu}.  
 Therefore, Theorem \ref{MainThm} would follow under a suitable (upper) bound for the mass. Using the fact that the sectional curvature is positive and the normalization (\ref{normal}), the Hessian Comparison Theorem can be used to build a supersolution of the equation satisfied by $G$.  This gives us an upper bound for $G$, hence on the mass;  precisely we prove 
\begin{align} \label{mup}
m(h) \leq \frac{45}{7} < \frac{22}{3}, 
\end{align}
hence $g$ is round.

\vspace{-0.2cm}

\medspace 

\subsection{Acknowledgements}  M.G. is partially supported by the Simons Foundation Travel Support for Mathematicians program, Award ID: SFI-MPS-TSM-00014122.  A.M. is a member of GNAMPA, as part of INdAM, and is supported by the project {\em Ricerca di base} from Scuola Normale Superiore. 

\medspace 

\vspace{-0.2cm}

\subsection{Use of AI Disclosure}   The proof of (\ref{mup}) was found by the authors with the help of ChatGPT.  As part of a continuing effort to apply the mass estimates in our paper \cite{GM}, we prompted ChatGPT to look for upper bounds of the Green's function of an Einstein metric under various curvature conditions, and (by trial and error) discovered that the assumption of positive sectional curvature led to a relatively elementary construction.   During the writing process the preprint \cite{ChengL} was posted on the archive.  Since our proof is based on a very different idea, and the work was essentially simultaneous, we decided to finish the manuscript.

\section{Distance estimates}\label{sec:comparison}

Throughout this section, $(M^4,g)$ is as in 
of Theorem \ref{MainThm} and satisfies (\ref{normal}).  As explained in the Introduction, we fix  $p\in M$, and let 
$r=d_{g}(p,\cdot)$. We first derive a lower bound for
$\Delta r$ via the Hessian comparison theorem.  

\begin{lemma}\label{lem:distance} The distance function is
smooth on $B_{\pi/2}(p)\setminus\{p\}$ and in a neighborhood of $\{ r = \pi/2 \}$.  Also, along every radial geodesic,
\begin{equation}\label{eq:distance}
\frac2r+\sqrt3\cot(\sqrt3r)
\le\Delta r\le3\cot r,
\qquad 0<r\le\frac{\pi}{2}.
\end{equation}
\end{lemma}

\begin{proof} For $q \in M$, let $u,v$ be unit tangent vectors with $g(u,v) = 0$, and choose $\{e_3,e_4\}$ so that $\{ u, v, e_3, e_4 \}$ is an orthonormal basis of $T_qM$.   Then  (\ref{normal}) gives 
\begin{equation}\label{eq:sectional-sum}
3=\Ric(u,u)=K(u,v)+K(u,e_3)+K(u,e_4),
\end{equation}
where $K(u,v)$ denotes the sectional curvature of the plane spanned by $u$ and $v$, etc.  
Since every sectional curvature is positive, it follows that $K(u,v) \leq 3$.  By Klingenberg's injectivity-radius estimate (see \cite{Klingenberg}), 
\begin{equation}\label{eq:injectivity}
\inj(M,g)\ge\frac{\pi}{\sqrt{K_{\max}}}
\ge\frac{\pi}{\sqrt3}>\frac{\pi}{2}.
\end{equation}
In particular, $B_{\pi/2}(p) \setminus \{ p \}$ is contained in a normal coordinate neighborhood of $p$, so $r$ is smooth there. 

Given a point $q \in B_{\pi/2}(p) \setminus \{ p \}$ let $\lambda_i$ be the non-radial eigenvalues of $\nabla^2 r \vert_q$.  Since $0 < K \leq 3$, the 
Hessian comparison theorem implies 
\begin{equation}\label{eq:hessian-eigenvalues}
\sqrt3\cot(\sqrt3r)\le\lambda_i\le\frac1r,\quad 0<r\le\frac{\pi}{2},
\quad i=1,2,3.
\end{equation}
This implies, for each $i$, 
\begin{equation}\label{eq:eigenvalue-product}
\left(\lambda_i-\sqrt3\cot(\sqrt3r)\right)
\left(\lambda_i-\frac1r\right)\le0.
\end{equation}
Expanding and summing over $i$, we obtain (at $q$) 
\begin{equation}\label{eq:hessian-square}
|\nabla^2r|^2
\le\left(\sqrt3\cot(\sqrt3r)+\frac1r\right)\Delta r
-\frac{3\sqrt3}{r}\cot(\sqrt3r).
\end{equation}

The normalization (\ref{normal}) and the Riccati equation along a unit-speed radial geodesic imply 
\begin{equation}\label{eq:trace-riccati}
\partial_r\Delta r
=-|\nabla^2r|^2-\Ric(\partial_r,\partial_r)
=-|\nabla^2r|^2-3.
\end{equation}
Substituting \eqref{eq:hessian-square} into \eqref{eq:trace-riccati}
gives  
\begin{equation}\label{eq:trace-linear-inequality}
\partial_r\Delta r
+\left(\frac1r+\sqrt3\cot(\sqrt3r)\right)\Delta r
\ge\frac{3\sqrt3}{r}\cot(\sqrt3r)-3.
\end{equation}

Note that $f(r) = 2/r+\sqrt3\cot(\sqrt3r)$ is a solution of 
\begin{equation}\label{eq:equality}
\partial_r f(r)
+\left(\frac1r+\sqrt3\cot(\sqrt3r)\right)f(r)
=\frac{3\sqrt3}{r}\cot(\sqrt3r)-3.
\end{equation}
Subtracting \eqref{eq:equality} from 
\eqref{eq:trace-linear-inequality} gives 
\begin{equation}\label{eq:difference-inequality}
\begin{aligned}
&\frac{d}{dr}\left(\Delta r-\frac2r-\sqrt3\cot(\sqrt3r)\right)\\
&\quad+\left(\frac1r+\sqrt3\cot(\sqrt3r)\right)
\left(\Delta r-\frac2r-\sqrt3\cot(\sqrt3r)\right)\ge0,
\end{aligned}
\end{equation}
hence 
\begin{equation}\label{eq:weighted-monotonicity}
\frac{d}{dr}\left[
r\sin(\sqrt3r)
\left(\Delta r-\frac2r-\sqrt3\cot(\sqrt3r)\right)
\right]\ge0.
\end{equation}

To obtain a comparison we need to compute the limit of
the quantity in square brackets as $r \to 0$.  Using the asymptotic expansion of $\Delta r$ (see \cite{Gray}), we find 
\begin{equation}\label{eq:weighted-limit}
r\sin(\sqrt3r)
\left(\Delta r-\frac2r-\sqrt3\cot(\sqrt3r)\right)
=O(r^5), \quad r \to 0. 
\end{equation}
Integrating \eqref{eq:weighted-monotonicity} from $\varepsilon$ to $r$
and then letting $\varepsilon\downarrow0$ shows that the left-hand side
of \eqref{eq:weighted-limit} is nonnegative. Dividing by $r\sin(\sqrt3r)$ gives the lower bound in \eqref{eq:distance}.
The upper bound in \eqref{eq:distance} follows from the Laplacian comparison (since $\Ric = 3g$). 
\end{proof}

\begin{lemma}\label{lem:scalar}
For $0<r\le\pi/2$,
\begin{equation}\label{eq:scalar-bound}
0\le
\frac{2\cos r}{\sin^3r}
\left(3\cot r-\frac2r-\sqrt3\cot(\sqrt3r)\right)
\le\frac47.
\end{equation}
\end{lemma}

\begin{proof}
For $0<r<\pi/2$, define 
\begin{equation}\label{eq:lem22-functions}
u=\frac1r-\cot r,\qquad
D=3\cot r-\frac2r-\sqrt3\cot(\sqrt3r),\qquad
V=\frac{2\sin^3r}{7\cos r}.
\end{equation}
It suffices to prove $0\le D\le V$.
Differentiating and using the fact that $\cot r=1/r-u$ and $\sqrt3\cot(\sqrt3r)=1/r-3u-D$, we find 
\begin{equation}\label{eq:lem22-ode}
D'=D^2+\left(6u-\frac2r\right)D+6u^2.
\end{equation}
By Taylor's theorem, 
\begin{equation}\label{eq:lem22-initial}
D=\frac{2}{15}r^3+O(r^5),\qquad
V=\frac27r^3+O(r^5).
\end{equation}
Thus $0<D<V$ for $r> 0$ small enough. 

We now prove two claims: 

\begin{claim} \label{TClaim1}  $u$ satisfes 
\begin{equation}\label{eq:lem22-u-bound}
0<u\le\frac{2\sin r}{3(1+\cos r)}.
\end{equation}
\end{claim}

\begin{proof}[Proof of the Claim]
The lower bound follows from the fact that $\sin r-r\cos r>0$ for $0<r\le\pi/2$.  To prove the upper bound, observe that the function $F(r)=r(2+\cos r)-3\sin r$ satisfies
$F(0)=F'(0)=0$ and $F''(r)>0$ for $0<r\le\pi/2$.  By convexity, it follows that 
\begin{align} \label{218b}
r ( 2 + \cos r) - 3 \sin r \geq 0. 
\end{align}
Rearranging this inequality gives
$1/r\le(2+\cos r)/(3\sin r)$.
Subtracting $\cot r$, we obtain
$u\le 2(1-\cos r)/(3\sin r)
=2\sin r/[3(1+\cos r)]$,
and the upper bound in (\ref{eq:lem22-u-bound}) follows. 
\end{proof}

\begin{claim} \label{TClaim2}  $V$ is a strict supersolution of
\eqref{eq:lem22-ode}. 
\end{claim}

\begin{proof}[Proof of the Claim] Fix $0<r<\pi/2$. We will use the following two estimates:
\begin{align} \begin{split} \label{ID1} 
-\frac{8\sin^3r}{7\cos r}\,u
&\ge
-\frac{16\sin^4r}{21\cos r(1+\cos r)},\\
-6u^2
&\ge
-\frac{8\sin^2r}{3(1+\cos r)^2}.
\end{split}
\end{align} 
Both follow from Claim \ref{TClaim1}.  The first follows by multiplying (\ref{eq:lem22-u-bound}) by $-8\sin^3r/(7\cos r)$.
The second is immediate.

Since $V'=V(3\cot r+\tan r)$ and $1/r=u+\cot r$,
we have
\begin{equation}\label{eq:claim24-expansion}
\begin{aligned}
&V'-V^2-\left(6u-\frac2r\right)V-6u^2\\
&\quad=
V(3\cot r+\tan r)-V^2-6uV
+2V(u+\cot r)-6u^2\\
&\quad=
V(5\cot r+\tan r)-V^2-4Vu-6u^2\\
&\quad=
\frac{10\sin^2r}{7}
+\frac{2\sin^4r}{7\cos^2r}
-\frac{4\sin^6r}{49\cos^2r}
-\frac{8\sin^3r}{7\cos r}\,u-6u^2,
\end{aligned}
\end{equation}
where in the last line we used $V=2\sin^3r/(7\cos r)$.  Applying the inequalities in
(\ref{ID1}) to the last two terms in \eqref{eq:claim24-expansion} we have 
\begin{equation}\label{eq:claim24-lower-bound}
\begin{aligned}
&V'-V^2-\left(6u-\frac2r\right)V-6u^2\\
&\quad\ge
\frac{10\sin^2r}{7}
+\frac{2\sin^4r}{7\cos^2r}
-\frac{4\sin^6r}{49\cos^2r}\\
&\qquad\quad
-\frac{16\sin^4r}{21\cos r(1+\cos r)}
-\frac{8\sin^2r}{3(1+\cos r)^2}\\
&\quad=
\frac{\sin^2r}{147c^2(1+c)^2}
\Bigl(30-52c-282c^2+496c^3\\
&\hspace{48mm}
+292c^4-24c^5-12c^6\Bigr),
\end{aligned}
\end{equation}
where $c=\cos r\in(0,1)$, and we have also used 
$\sin^2r=1-c^2$. 

The term in parentheses is positive since $0 < c < 1$ and 
\begin{equation}\label{eq:claim24-polynomial}
\begin{aligned}
&30-52c-282c^2+496c^3+292c^4-24c^5-12c^6\\
&\quad=
\left(c-\frac38\right)^2
(256c^2+688c+198)
+\frac{69-8c}{32}
+12c^4(1-c)(3+c).
\end{aligned}
\end{equation}
Therefore
\begin{equation}\label{eq:claim24-strict}
V'>V^2+\left(6u-\frac2r\right)V+6u^2
\qquad (0<r<\pi/2),
\end{equation}
and $V$ is a strict supersolution of (2.15).

We remark that numerically, the upper bound in \eqref{eq:scalar-bound} 
is approximately $0.419$. 
\end{proof}

\medskip

Finally, we claim that $D(r) > 0$ on $0 < r < \pi/2$.  Note that \eqref{eq:lem22-initial} implies that $D(r)>0$ for sufficiently small $r > 0$. 
Suppose $D$ vanishes somewhere
in $(0,\pi/2)$, and let $r_0$ be the smallest positive number such that 
$D(r_0)=0$.
Then $D(r)>0$ for $0<r<r_0$, so $D'(r_0)\le0$.
However, \eqref{eq:lem22-ode} gives
$D'(r_0)=6u(r_0)^2>0$, a contradiction.
Thus $D(r)>0$ throughout $(0,\pi/2)$.   A similar argument shows $V - D > 0$.  
Multiplying the inequality $0 < D < V$ by $2 \cos r/\sin^3r$ proves \eqref{eq:scalar-bound}
on this interval.   By continuity, the inequality holds at $r=\pi/2$.
\end{proof}

\section{Green function comparison and the mass}\label{sec:green}

The main result of this section is the following: 

\begin{theorem}\label{thm:mass}
Let $(M^4,g)$ satisfy the hypotheses of Theorem \ref{MainThm}, and the normalization (\ref{normal}).  Given $p\in M$, let 
$G = G_p$ denote the Green's function of $L = \Delta -2,$ normalized by 
\begin{align} 
G = r^{-2} + A + O(r), 
\end{align}
where $r$ denotes the distance to $p$.  Then $G$ satisfies the inequality
\begin{equation}\label{eq:green-upper}
G(x)\le
\begin{cases}
\csc^2 r(x)+\dfrac27,&0<r(x)\le\pi/2,\\[4pt]
\dfrac97,&r(x)\ge\pi/2.
\end{cases}
\end{equation}
Consequently, the constant $A$ in the expansion (\ref{Gnorm}) satisfies
\begin{equation}\label{eq:theorem-conclusion}
A\le\frac{13}{21},
\end{equation}
and the mass of the metric $h = G^2 g$ satisfies 
\begin{align} \label{massupper}
 m(h)\le\frac{45}{7}.
\end{align}
\end{theorem}

\medskip 

\begin{remark} The relationship between the mass $m(h)$ and the constant $A$ was proved by Viaclovsky in \cite{V}.  The point is that conformal normal coordinates are not needed thanks to the Einstein condition. 
\end{remark}

\medskip 

\begin{proof}
On $M\setminus\{p\}$ define 
\begin{equation}\label{eq:barrier}
\Psi(x)=
\begin{cases}
\csc^2 r(x),&0<r(x)\le\pi/2,\\
1,&r(x)\ge\pi/2.
\end{cases}
\end{equation}
Note that $\Psi(\pi/2) = 1$ and $(\partial_r \Psi)(\pi/2) = 0$.  Also, the cut locus lies in the set $\{ r > \pi/2 \}$, where $\Psi$ is constant. 

In $B_{\pi/2}(p)$ we have 
\begin{equation}\label{eq:barrier-derivatives}
\frac{d}{dr}\csc^2r=-\frac{2\cos r}{\sin^3r},
\qquad
\frac{d^2}{dr^2}\csc^2r=\frac2{\sin^2r}+\frac{6\cos^2r}{\sin^4r}.
\end{equation}
Since $\Psi$ is radial, these imply  
\begin{equation}\label{eq:barrier-laplacian}
\begin{aligned}
(\Delta-2)\Psi
&=\frac2{\sin^2r}+\frac{6\cos^2r}{\sin^4r}
-\frac{2\cos r}{\sin^3r}\Delta r-\frac2{\sin^2r}\\
&=\frac{2\cos r}{\sin^3r}(3\cot r-\Delta r),
\qquad 0<r<\pi/2.
\end{aligned}
\end{equation}
By Lemmas \ref{lem:distance} and \ref{lem:scalar}, and the fact that $\cos r\ge0$ for $0 < r < \pi/2$, we have 
\begin{equation}\label{eq:error}
0\le (\Delta-2)\Psi
\le\frac{2\cos r}{\sin^3r}
\left(3\cot r-\frac2r-\sqrt3\cot(\sqrt3r)\right)
\le\frac47
\qquad(0<r<\pi/2).
\end{equation}
Since $\Psi \equiv 1$ for $r > \pi/2$, 
\begin{equation}\label{eq:exterior-error}
(\Delta-2)\Psi=-2\qquad(r>\pi/2).
\end{equation}

Near $p$, the expansion of $\csc^2r$ gives
\begin{equation}\label{eq:barrier-expansion}
\Psi=r^{-2}+\frac13+O(r^2),\qquad
\partial_r\Psi=-2r^{-3}+O(r).
\end{equation}
Since the area of a small geodesic sphere is
$2\pi^2\varepsilon^3(1+O(\varepsilon^2))$, 
\begin{equation}\label{eq:barrier-flux}
\lim_{\varepsilon\downarrow0}
\int_{\partial B_\varepsilon(p)}\partial_r\Psi\,d\sigma_{g}
=-4\pi^2.
\end{equation}
\medskip

\begin{claim} \label{DisClaim} In the sense of distributions,
\begin{equation}\label{eq:distributional-barrier}
(\Delta-2)\Psi\le-4\pi^2\delta_p+\frac47
\quad\hbox{on }M.
\end{equation}
\end{claim}

\begin{proof}[Proof of the claim]
Define the bounded function $f$ almost everywhere on $M$ by
\begin{equation*}
\label{eq:claim-Psi-density}
f(x)=
\begin{cases}
\displaystyle
\frac{2\cos r}{\sin^3r}
\bigl(3\cot r-\Delta r\bigr),
& 0<r<\pi/2,\\[6pt]
-2,
& r>\pi/2.
\end{cases}
\end{equation*}
Equations \eqref{eq:barrier-laplacian} and \eqref{eq:exterior-error} show that
$(\Delta-2)\Psi=f$ pointwise on
$M\setminus\bigl(\{p\}\cup\{r=\pi/2\}\bigr)$,
where $\Psi$ is smooth.
Moreover, \eqref{eq:error} gives $0\leq f\leq4/7$ inside $B_{\pi/2}(p)$, 
while $f=-2$ outside the ball.  It follows that $f\leq4/7$ almost everywhere on $M$.

Let $\varphi\in C^\infty(M)$ be a non-negative test function, and $0<\varepsilon<\pi/2$. Set
$\Omega_\varepsilon
=B_{\pi/2}(p)\setminus\overline{B_\varepsilon(p)}$
and $E=M\setminus\overline{B_{\pi/2}(p)}$.  Applying Green's formula on $\Omega_\varepsilon$ gives
\begin{align*}
\int_{\Omega_\varepsilon}
\Psi(\Delta-2)\varphi\,dv_g
={}&\int_{\Omega_\varepsilon}f\varphi\,dv_g\\
&+\int_{\partial B_{\pi/2}(p)}
\left(\Psi\,\partial_r\varphi
-\varphi\,\partial_r\Psi\right)d\sigma_g\\
&+\int_{\partial B_\varepsilon(p)}
\left(-\Psi\,\partial_r\varphi
+\varphi\,\partial_r\Psi\right)d\sigma_g\\
={}&\int_{\Omega_\varepsilon}f\varphi\,dv_g
+\int_{\partial B_{\pi/2}(p)}
\partial_r\varphi\,d\sigma_g\\
&+\int_{\partial B_\varepsilon(p)}
\left(-\Psi\,\partial_r\varphi
+\varphi\,\partial_r\Psi\right)d\sigma_g.
\end{align*}
Similarly, Green's formula on $E$ gives
\begin{align*}
\int_E\Psi(\Delta-2)\varphi\,dv_g
={}&\int_E f\varphi\,dv_g\\
&+\int_{\partial B_{\pi/2}(p)}
\left(-\Psi\,\partial_r\varphi
+\varphi\,\partial_r\Psi\right)d\sigma_g\\
={}&\int_E f\varphi\,dv_g
-\int_{\partial B_{\pi/2}(p)}
\partial_r\varphi\,d\sigma_g.
\end{align*}
Upon adding, the integrals of
$\partial_r\varphi$ over $\partial B_{\pi/2}(p)$
cancel and we obtain
\begin{align}
\label{eq:claim-Psi-Green}
\begin{split}
&\int_{M\setminus B_\varepsilon(p)}
\Psi(\Delta-2)\varphi\,dv_{g}\\
&\qquad=
\int_{M\setminus B_\varepsilon(p)}
f\varphi\,dv_{g}
+\int_{\partial B_\varepsilon(p)}
\left(
-\Psi\,\partial_r\varphi
+\varphi\,\partial_r\Psi
\right)d\sigma_{g}.
\end{split}
\end{align}

By \eqref{eq:barrier-expansion}, $\Psi=\varepsilon^{-2}+O(1)$ and
$\partial_r\Psi=-2\varepsilon^{-3}+O(\varepsilon)$
on $\partial B_\varepsilon(p)$.
Moreover,
$\varphi=\varphi(p)+O(\varepsilon)$,
$\partial_r\varphi=O(1)$, and
$\operatorname{Area}_{g}(\partial B_\varepsilon(p))
=2\pi^2\varepsilon^3(1+O(\varepsilon^2))$.
Consequently,
\begin{equation*}
\label{eq:claim-Psi-boundary}
\begin{aligned}
\int_{\partial B_\varepsilon(p)}
\Psi\,\partial_r\varphi\,d\sigma_{g}
&=O(\varepsilon),\\
\int_{\partial B_\varepsilon(p)}
\varphi\,\partial_r\Psi\,d\sigma_{g}
&=-4\pi^2\varphi(p)+O(\varepsilon).
\end{aligned}
\end{equation*}

The function $\Psi=O(r^{-2})$ is locally integrable
in dimension four, and $f$ is bounded, hence we can take the limit $\varepsilon\downarrow0$ in
(\ref{eq:claim-Psi-Green}).
For every nonnegative $\varphi\in C^\infty(M)$,
the resulting identity and the bound $f\leq4/7$ give
\begin{equation*}
\label{eq:claim-Psi-distribution}
\begin{aligned}
\bigl\langle(\Delta-2)\Psi,\varphi\bigr\rangle
&=\int_M\Psi(\Delta-2)\varphi\,dv_{g}\\
&=-4\pi^2\varphi(p)+\int_M f\varphi\,dv_{g}\\
&\leq-4\pi^2\varphi(p)
+\frac47\int_M\varphi\,dv_{g}.
\end{aligned}
\end{equation*}
\end{proof}

Subtracting the distributional identity $(\Delta-2)G=-4\pi^2\delta_p$ from the distributional inequality
\eqref{eq:distributional-barrier} cancels the Dirac masses at $p$
and gives $(\Delta-2)(\Psi-G)\le4/7$ on $M$.
Since $(\Delta-2)(2/7)=-4/7$, adding the constant $2/7$ to
$\Psi-G$ yields
\begin{equation}\label{eq:weak}
(\Delta-2)\left(\Psi+\frac27-G\right)\le0
\quad\hbox{on }M
\end{equation}
in the distributional sense.  However, the expansions (\ref{Gnorm}) and (\ref{eq:barrier-expansion}) imply
\begin{equation}\label{eq:regular-difference}
\Psi+\frac27-G
=\frac{13}{21}-A+O(r),
\end{equation}
hence $\Psi+\frac27-G \in H^1(M)$ (see \cite{V}, Proposition 2.1).  Then the weak maximum principle implies \eqref{eq:green-upper}. 

Taking $r \to 0$ in \eqref{eq:regular-difference} gives
\begin{equation}\label{eq:constant-term-bound}
0\le\lim_{x\to p}\left(\Psi(x)+\frac27-G(x)\right)
=\frac{13}{21}-A,
\end{equation}
and since $m(g) = 12 A - 1$ the mass estimate follows. 
\end{proof}

\section{Proof of Theorem \ref{MainThm}}\label{sec:rigidity}

\begin{proof}[Proof of Theorem \ref{MainThm}]
By Synge's theorem, $M$ is simply connected; in particular, $b_1(M)=0$.
Let $b_2^+$ and $b_2^-$ be the dimensions of the spaces of self-dual and anti-self-dual harmonic two-forms.  Then 
\begin{equation}\label{eq:betti-topology}
\chi(M)=2+b_2^++b_2^-,\qquad
\tau(M)=b_2^+-b_2^-.
\end{equation}
If $g$ is self-dual or anti-self-dual, Hitchin's theorem [16]
implies that $(M,g)$ is homothetically isometric to the round
$S^4$ or to $\pm\mathbb{CP}^2$ with the Fubini--Study metric.
This proves the theorem in this case.

Suppose, therefore, that $g$ is neither self-dual nor
anti-self-dual. Combining the Gursky--LeBrun inequality
(1.2) with the topological assumption (1.1), we obtain
\begin{equation}
    \frac{6}{5}\chi(M)
    < 2\chi(M)-3|\tau(M)|
    \le 4.
\end{equation}
Thus $\chi(M)<10/3$, and another application of (1.2) gives
\begin{equation}
    |\tau(M)|
    < \frac{4}{15}\chi(M)
    < \frac{8}{9}.
\end{equation}
Since $\tau(M)$ is an integer, it follows that $\tau(M)=0$.
Assumption (1.1) now gives $\chi(M)\le 2$, while (4.1)
gives $\chi(M)\ge 2$. Consequently,
\begin{equation}
    \chi(M)=2,
    \qquad
    b_2(M)=0.
\end{equation}

If $g$ were not round, \cite[Theorem~1.5]{GM} would imply
\begin{equation}\label{eq:gap-contradiction}
m(h)\ge\frac{60}{4}-9+\frac43
=\frac{22}{3}>\frac{45}{7}, 
\end{equation}
contradicting Theorem~\ref{thm:mass}. Hence $g$ is round and $M = S^4$.  
\end{proof}


\end{document}